\documentclass[11pt]{amsart}

\usepackage[margin=1in]{geometry}
\usepackage{amsmath,amssymb,amsthm,mathtools}
\usepackage{color}
\usepackage{graphicx,float}
\usepackage{enumitem}
\usepackage[T1]{fontenc}
\usepackage{lmodern}
\usepackage[hidelinks]{hyperref}

\setlist[enumerate]{itemsep=0.35em,topsep=0.35em}

\newtheorem{theorem}{Theorem}
\newtheorem{conjecture}[theorem]{Conjecture}
\newtheorem{proposition}[theorem]{Proposition}
\newtheorem{lemma}[theorem]{Lemma}
\theoremstyle{remark}

\DeclareMathOperator{\Sym}{\mathsf{S}^2}
\DeclareMathOperator{\GL}{GL}
\DeclareMathOperator{\diag}{diag}
\DeclareMathOperator{\tr}{tr}
\DeclareMathOperator{\sgn}{sgn}

\newcommand{\tp}{{\scriptscriptstyle\mathsf{T}}}

\begin{document}

\title{Pierce--Birkhoff conjecture is false}
\author[Z.~Lai]{Zehua~Lai}
\address{Institute of Applied Mathematics, Academy of Mathematics and Systems Science, Chinese Academy of Sciences, Beijing 100190, China}
\email{lzh@amss.ac.cn}
\author[L.-H.~Lim]{Lek-Heng Lim}
\author[J.~Ren]{Junyu Ren}
\address{Computational and Applied Mathematics Initiative, University of Chicago, Chicago, IL 60637}
\email{lekheng, junyuren@uchicago.edu}
\date{September 9, 2026}

\begin{abstract}
We provide a counterexample to the Pierce--Birkhoff conjecture: a continuous function that is piecewise quadratic on a finite collection of semialgebraic sets that partition $\mathbb{R}^n$ but that cannot be expressed as a finite lattice combination of polynomials. 
The counterexample was found with the assistance of our multi-agent, multi-model setup that chains together GPT 5.6, GPT 6, Claude Opus 5, and Claude Fable 5.1.
\end{abstract}

\maketitle

The Pierce--Birkhoff conjecture states that every continuous piecewise polynomial function $f : \mathbb{R}^n \to \mathbb{R}$ on a semialgebraic partition of $\mathbb{R}^n$ is a finite lattice combination of polynomials, i.e., the functions generated by addition, multiplication, and maximization of polynomial functions on $\mathbb{R}^n$. To be precise, such a function takes the form
\begin{equation}\label{eq:def}
f = g_1 1_{\Pi_1} + \dots +  g_r 1_{\Pi_r}
\end{equation}
where $g_1,\dots,g_r \in \mathbb{R}[x_1,\dots,x_n]$, and
\begin{equation}\label{eq:part}
\mathbb{R}^n = \Pi_1 \cup \dots \cup \Pi_r
\end{equation}
is a partition into semialgebraic sets $\Pi_1,\dots,\Pi_r$. Here $1_{\Pi}$ denotes the indicator function. A function as defined by \eqref{eq:def} and \eqref{eq:part} is sometimes called a semialgebraic spline \cite{dipasquale2017}, a generalization of splines \cite{Chui}, the special case when $\Pi_1,\dots,\Pi_r$ are polytopes.

So a precise statement of the Pierce--Birkhoff conjecture is as follows:
\begin{conjecture}[Pierce--Birkhoff]\label{conj:PB}
Any semialgebraic spline $f : \mathbb{R}^n \to \mathbb{R}$ can be expressed as
\begin{equation}\label{eq:PB}
f=\max_{i=1,\dots,p}\min_{j=1,\dots,p'} f_{ij}
\end{equation}
with finitely many polynomials $f_{ij} \in \mathbb{R}[x_1,\dots,x_n]$, $i =1,\dots,p$, $j=1,\dots,p'$.
\end{conjecture}
The conjecture was first formally stated in  \cite{HI} but has roots in Birkhoff and Pierce's earlier work \cite{BP}. It is only known to be true for $n \le 2$ \cite{Mahe1984}. We have proved this conjecture in \cite{LL} (a purely human endeavor with no AI usage) for splines, i.e., when the semialgebraic sets $\Pi_1,\dots,\Pi_r$ are defined by \emph{linear} polynomials. The counterexample constructed below shows that this result does not extend to arbitrary semialgebraic partitions, even for piecewise quadratic functions.

The role of large language models in our discovery and verification process is described in Section~\ref{sec:ai-usage}. The counterexample in this article is notably not found with any single large language model (we tried doing so without success) but requires harnessing multiple systems from both OpenAI and Anthropic in a more intricate setup.

The $n = 30$ counterexample presented below is chosen because it can be readily ``checked by hand.'' Nevertheless, our setup has also generated counterexamples of dimensions $n = 5$, $6$, $7$, $22$, $24$, $72$ although they are harder for a human reader to verify. The counterexamples for $n =5$, $6$, $7$, $30$ (the one described in this article), and $72$ have also been verified through Lean-formalization: \url{https://github.com/7pocheR/Pierce-Birkhoff}.

\section{The 30-dimensional function}

We begin by defining the $30$-dimensional space
\[
\mathbb{V} \coloneqq \Sym(\mathbb{R}^5)\times\Sym(\mathbb{R}^5).
\]
A point of $\mathbb{V}$ is a pair of symmetric matrices $(X,Y)$. Set
\[
P=XY,
\]
noting that $P$ need not be a symmetric matrix even though $X$ and $Y$ are. For $i\ne j$ define the homogeneous quadratic forms
\[
q_{ij,\pm}(X,Y)=P_{ii}\pm 4 P_{ij}.
\]
There are $5\cdot4\cdot2=40$ such forms. Set
\[
h(X,Y)=\min_{i\ne j,\, \pm}q_{ij, \pm}(X,Y).
\]
Finally define
\begin{equation}\label{eq:counterexample}
f(X,Y)=
\begin{cases}
h(X,Y),& h(X,Y)>0\text{ and }\det X>0,\\[2mm]
0,&\text{otherwise}.
\end{cases}
\end{equation}

\begin{lemma}
If $h(X,Y)>0$, then $P=XY$ is invertible. In particular $X$ and $Y$ are invertible.
\end{lemma}
\begin{proof}
For every $i\ne j$, positivity of both $q_{ij,+}$ and $q_{ij,-}$ gives
\[
P_{ii}>4|P_{ij}|.
\]
Hence $P_{ii}>0$ and
\[
\sum_{j\ne i}|P_{ij}|<P_{ii}
\]
for every row $i$. Thus $P$ is strictly diagonally dominant by rows. Therefore $P$ is invertible. Since $P=XY$, both factors are invertible.
\end{proof}

Recall that a function is positively homogeneous of degree $d$ if $f(t x) = t^d f(x)$ for all $t > 0$.
\begin{proposition}
The function $f$ is a continuous, positively homogeneous of degree two, and piecewise quadratic polynomial.
\end{proposition}
\begin{proof}
On the open set $\{h>0\}$, the previous lemma implies $\det X\ne0$, so $\sgn(\det X)$ is locally constant. This implies $f$ is  continuous.

By definition, we see that
\[
q_{ij,\pm}(tX,tY)=t^2q_{ij,\pm}(X,Y),
\qquad
\det(tX)=t^5\det X
\]
for $t>0$. Hence
\[
f(tX,tY)=t^2f(X,Y). \qedhere
\]
\end{proof}

\section{A property of lattice combination of polynomials}

To prove $f$ is indeed a counterexample, i.e., not a lattice combination of polynomials, we need the following property of lattice combination of polynomials.

\begin{lemma}[Separation by quadratic polynomials]
Let $g:\mathbb R^n\to\mathbb R$ satisfy
\[
g(tx)=t^2g(x)
\]
for all $t>0$. 
If $g$ has a finite lattice expression in arbitrary real polynomials, then there exists a finite family $S$ of real polynomials of degree at most two such that the signs of polynomials in $S$ determine the sign of $g$.
\end{lemma}
\begin{proof}
Write
\[
g=\max_{i=1,\dots,p}\min_{j=1,\dots,p'} g_{ij}.
\]
Decompose
\[
g_{ij}=c_{ij}+\ell_{ij}+Q_{ij}+\sum_{k\ge3}P_{ij,k}
\]
into homogeneous parts. Positive scaling commutes with $\min$ and $\max$, hence
\[
g(x)=\max_{i=1,\dots,p}\min_{j=1,\dots,p'} \frac{g_{ij}(tx)}{t^2}.
\]
As $t\downarrow0$, each entry has an extended-real limit
\[
u_{ij}(x)=
\begin{cases}
\sgn(c_{ij})\cdot\infty&\text{if } c_{ij}\ne0,\\
\sgn(\ell_{ij}(x))\cdot\infty&\text{if } c_{ij}=0,\ \ell_{ij}(x)\ne0,\\
Q_{ij}(x)&\text{if } c_{ij}=0,\ \ell_{ij}(x)=0.
\end{cases}
\]
Terms of degree at least three disappear after division by $t^2$. Since finite $\min$ and $\max$ are continuous on the extended real line,
\[
g(x)=\max_{i=1,\dots,p}\min_{j=1,\dots,p'} u_{ij}(x).
\]
The sign map preserves $\min$ and $\max$. Thus $\sgn g(x)$ is determined by the signs of $u_{ij}(x)$, and thus the signs of $\ell_{ij}(x)$ and $Q_{ij}(x)$.
\end{proof}

Therefore, to prove $f$ is not a lattice combination, it is enough to show that no finite family of polynomials with degree at most two determines $\sgn f$.

\section{The boundary orbit and its quadratic equations}

Define
\[
D=\diag(1,1,0,0,0),
\qquad
E=\diag(0,0,1,1,0),
\]
and consider the orbit
\begin{equation}
\mathcal O=
\bigl\{
\bigl(UDU^\tp ,U^{-\tp} EU^{-1}\bigr):U\in\GL_5(\mathbb R)
\bigr\}\subset \mathbb{V}.
\label{eq:orbit}
\end{equation}
For every $(X,Y)\in\mathcal O$,
\[
XY=UDU^\tp U^{-\tp} EU^{-1}=UDEU^{-1}=0,
\]
and similarly $YX=0$.

\begin{lemma}[Quadratic orbit lemma]\label{lem:quadratic-orbit}
A real polynomial $p(X,Y)$ of total degree at most two vanishes on $\mathcal O$ if and only if
\[
p(X,Y)=L(XY)
\]
for some real linear functional $L$ on $\mathbb{R}^{5 \times 5}$.
\end{lemma}

\begin{proof}
The reverse implication is immediate from $XY=0$ on $\mathcal O$.

Now we prove that if $p$ vanishes on $\mathcal O$ and has degree at most 2, then $p$ has the desired form. First note that replacing $U$ by
\[
U\diag(a,a,b^{-1},b^{-1},1)
\]
replaces $(X,Y)$ by $(a^2X,b^2Y)$. Hence if $p$ vanishes on $\mathcal O$, then all bihomogeneous components of $p$ vanish on $\mathcal O$.

The projection of $\mathcal O$ to the first coordinate $X$ consists of positive semidefinite symmetric matrices of rank two. Its Euclidean closure contains every rank-one matrix $uu^\tp $. A linear form vanishing on all $uu^\tp $ is zero since matrices of the form $uu^\tp $ span the space $\Sym(\mathbb{R}^5)$.

If a homogeneous quadratic form $Q$ vanishes on the rank-two positive semidefinite locus, then it also vanishes on $uu^\tp $ and on $uu^\tp +vv^\tp $. Polarization gives
\[
\widetilde Q(uu^\tp ,vv^\tp )=0
\]
for all $u,v$. Again, since rank-one symmetric matrices span $\Sym(\mathbb{R}^5)$, $Q=0$. The same argument applies to pure $Y$ components. The constant component is also zero. Thus only a bilinear form $B(X,Y)$ may remain.

We next show that the Euclidean closure of $\mathcal O$ contains every pair
\[
(uu^\tp ,ss^\tp )\qquad\text{with }u^\tp s=0.
\]
Take
\[
U_\delta=\diag(1,\delta,1,\delta^{-1},1).
\]
Then
\[
U_\delta DU_\delta^\tp =E_{11}+\delta^2E_{22},
\qquad
U_\delta^{-\tp} EU_\delta^{-1}=E_{33}+\delta^2E_{44},
\]
which tends to $(E_{11},E_{33})$ as $\delta \to 0$. If $u^\tp s=0$, we can choose $W\in\GL_5(\mathbb R)$ with $We_1=u$ and $W^{-\tp} e_3=s$; applying the fixed change of basis $W$ gives the desired pair in the closure.

Hence
\[
B(uu^\tp ,ss^\tp )=0\qquad\text{whenever }u^\tp s=0.
\]
Set
\[
F(u,s)=B(uu^\tp ,ss^\tp ).
\]
This polynomial has bidegree $(2,2)$ and vanishes on the hypersurface $u^\tp s=0$. The polynomial $u^\tp s$ is irreducible, and its real zero set is Zariski dense in that hypersurface. Therefore
\[
F(u,s)=(u^\tp s)G(u,s)
\]
with $G$ of bidegree $(1,1)$. Thus $G(u,s)=u^\tp As$ for some matrix $A$, so
\[
F(u,s)=(u^\tp s)(u^\tp As).
\]
Define $L(M)=\tr(A^\tp M)$. Since
\[
(uu^\tp )(ss^\tp )=(u^\tp s)us^\tp ,
\]
we have
\[
B(uu^\tp ,ss^\tp )=L(uu^\tp ss^\tp ).
\]
Rank-one symmetric matrices span $\Sym(\mathbb{R}^5)$, so bilinearity extends the identity to all symmetric $X,Y$:
\[
B(X,Y)=L(XY). \qedhere
\]
\end{proof}

\section{The two perturbations}

Fix $U\in\GL_5(\mathbb R)$. For $\eta\in\{+1,-1\}$ and $\varepsilon>0$, define
\[
D_X^\eta=\diag(1,1,\varepsilon^2,\varepsilon^2,\eta\varepsilon),
\qquad
D_Y^\eta=\diag(\varepsilon^2,\varepsilon^2,1,1,\eta\varepsilon),
\]
and
\begin{equation}
X_\eta=UD_X^\eta U^\tp ,
\qquad
Y_\eta=U^{-\tp} D_Y^\eta U^{-1}.
\label{eq:perturbations}
\end{equation}

For both signs $\eta$, as $\varepsilon \to 0$, 
\[
(X_\eta,Y_\eta)\longrightarrow
(UDU^\tp ,U^{-\tp} EU^{-1})\in\mathcal O.
\]

Direct multiplication gives
\[
D_X^\eta D_Y^\eta
=\diag(\varepsilon^2,\varepsilon^2,\varepsilon^2,\varepsilon^2,\eta^2\varepsilon^2)
=\varepsilon^2I_5.
\]
Hence
\begin{equation}
X_\eta Y_\eta=\varepsilon^2I_5,
\qquad
Y_\eta X_\eta=\varepsilon^2I_5.
\label{eq:perturbation-products}
\end{equation}

On the other hand,
\[
\det D_X^\eta=\eta\varepsilon^5,
\]
so
\begin{equation}
\det X_\eta=(\det U)^2\eta\varepsilon^5.
\label{eq:perturbation-determinant}
\end{equation}
Therefore
\[
\det X_+>0,
\qquad
\det X_-<0.
\]
The determinant remembers the hidden sign.

Since $X_\eta Y_\eta=\varepsilon^2I_5$, all off-diagonal entries of $P$ vanish and all diagonal entries equal $\varepsilon^2$. Thus every defining quadratic satisfies
\[
q_{ij, \pm}(X_\eta,Y_\eta)=\varepsilon^2,
\]
and consequently
\[
h(X_\eta,Y_\eta)=\varepsilon^2.
\]
It follows from \eqref{eq:counterexample} and \eqref{eq:perturbation-determinant} that
\begin{equation}
f(X_+,Y_+)=\varepsilon^2,
\qquad
f(X_-,Y_-)=0.
\label{eq:perturbation-values}
\end{equation}

\begin{proposition}
For every finite family of real polynomials $p_1,\dots,p_M$ of degree at most two, there exist two points $z_+,z_-\in \mathbb{V}$ having identical signs for all $p_j$, but satisfying
\[
f(z_+)>0,
\qquad
f(z_-)=0.
\]
\end{proposition}

\begin{proof}
Split the $p_j$ into two classes.

If $p_j$ does not vanish identically on $\mathcal O$, then its pullback
\[
U\longmapsto p_j(UDU^\tp ,U^{-\tp} EU^{-1})
\]
is a nonzero rational function on $\GL_5(\mathbb R)$. A finite collection of such nonzero rational functions is simultaneously nonzero at some $U\in\GL_5(\mathbb R)$. Fix such a generic $U$ and write
\[
x=(UDU^\tp ,U^{-\tp} EU^{-1}).
\]
Then $p_j(x)\ne0$ for every polynomial in this first class. Since both perturbations tend to $x$, continuity gives the same strict sign for $p_j(X_+,Y_+)$ and $p_j(X_-,Y_-)$ once $\varepsilon$ is sufficiently small.

If $p_j$ vanishes identically on $\mathcal O$, Lemma~\ref{lem:quadratic-orbit} gives
\[
p_j(X,Y)=L_j(XY).
\]
Using \eqref{eq:perturbation-products},
\[
p_j(X_+,Y_+)=p_j(X_-,Y_-)=L_j(\varepsilon^2I_5)
\]
exactly.

There are finitely many $p_j$, so one sufficiently small $\varepsilon>0$ works for all of them.
\end{proof}

\begin{theorem}[The 30-dimensional lattice obstruction]
The function $f$ is a continuous, positively homogeneous of degree two, and piecewise quadratic polynomial. It admits no finite lattice expression in real polynomials of arbitrary degrees.
\end{theorem}

\section{Use of artificial intelligence}\label{sec:ai-usage}

The counterexample was obtained through a sustained investigation in a customized multi-agent, multi-model harness using models from the GPT (5.6 and 6) and Claude families (Opus 5 and Fable 5.1). Built around Codex and Claude Code, with additional consultations through ChatGPT, the harness supported persistent files and local code execution across sessions. Neither model family was assigned a predetermined role, such as construction or review. Roles were assigned dynamically as the investigation evolved, with both families constructing and reviewing arguments. The aim was to exploit diversity in mathematical ideas and intuitions, including different choices of techniques and different emphases on retrieved results and prior observations.

Human input provided high-level guidance on the mathematical significance of the different research targets. An orchestrator agent tracked progress, assigned parallel proof searches, counterexample searches, and reviews, and decided when to pursue or change direction, informed by cross-model discussion and fresh-context audits. It coordinated extensive exchanges of arguments, intermediate results, objections, and corrections across agents and model families. To limit bias from earlier suggestions or verdicts, selected investigators initially received only the problem statement, while selected reviewers received candidate arguments without prior review conclusions. A persistent record distinguished proved results, failed attempts, and unresolved reductions; executable algebraic checks were repeated on revised arguments. Figure~\ref{fig:ai-setup} summarizes the setup.

\begin{figure}[htb]
  \centering
  \includegraphics[width=\linewidth]{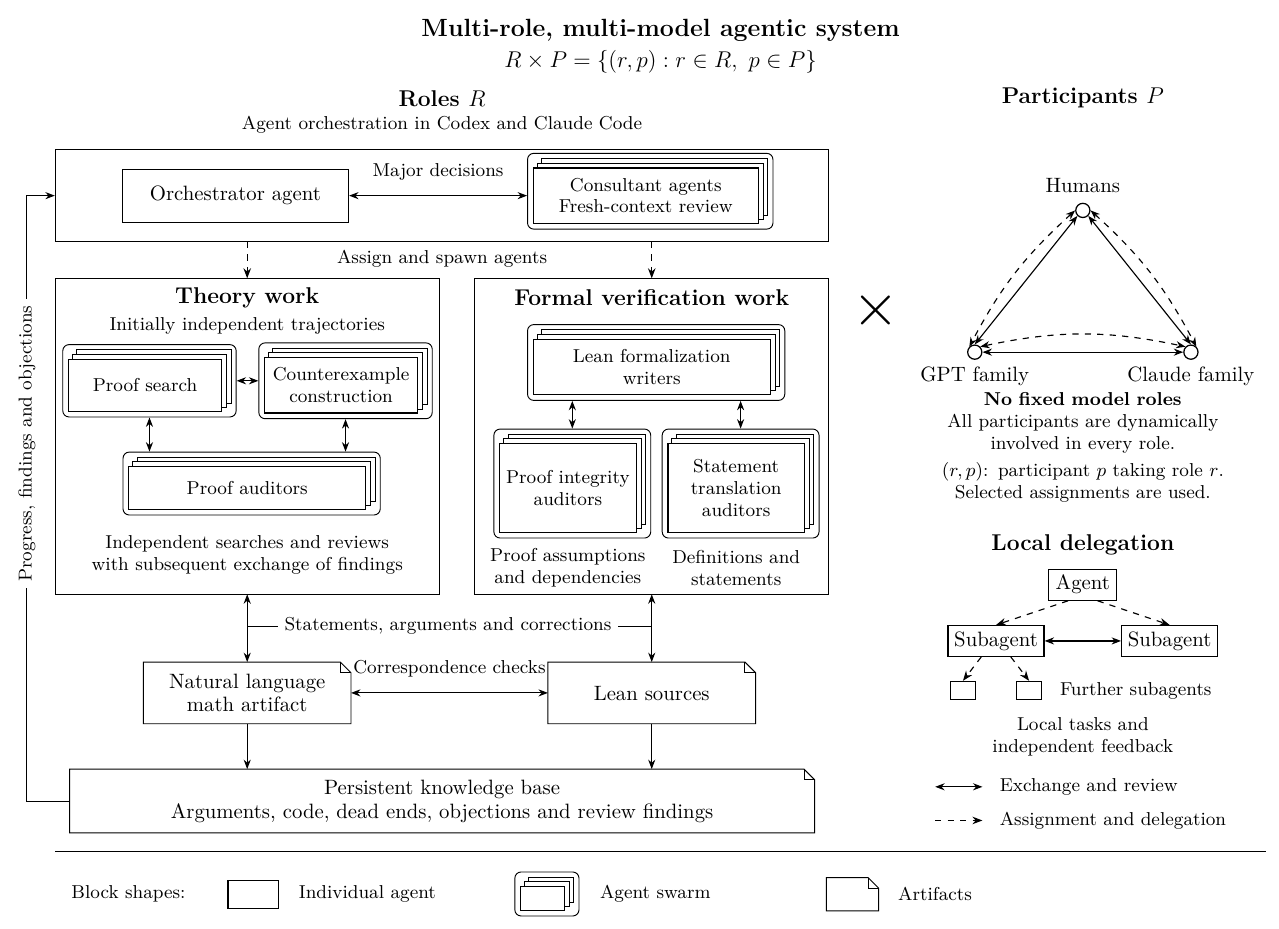}
  \caption[The multi-agent setup]{The multi-agent setup we used in this work.}
  \label{fig:ai-setup}
\end{figure}

Our investigation started with an attempt to extend results for hyperplane partitions to partitions defined by quadratic inequalities, alongside the unrestricted counterexample search. An auxiliary-degree reduction from the attempt to prove this extension supplied a counterexample criterion, and its radial argument was retained in the final proof. The matrix-pair construction in dimension~72 and the argument excluding representations of arbitrary degree emerged from GPT investigations. With further assistance from ChatGPT, we obtained the symmetric-matrix construction in dimension~30, simplified the proof, and prepared the present exposition.

\section{Insight in artificial intelligence}\label{sec:ai-app}

This article is slightly unusual in that while it is AI-assisted, it also sheds light on the AI technology that produces it. We showed in \cite{lai2024}  that a \emph{pure transformer}, i.e., with the ``add \& norm'' layers omitted, is a semialgebraic spline, and that if the Pierce--Birkhoff conjecture were true, then the converse would also hold. The counterexample here shows that this is not the case. Combined with the work in \cite{lai2024} and \cite{LL}, we now know that the inclusions are strict:
\[
\{ \text{splines} \} \quad \subsetneq \quad\{ \text{pure transformers} \} \quad \subsetneq \quad\{ \text{semialgebraic splines} \},
\]
i.e., every spline can be represented as a  ReLU-activated pure transformer, but there exist semialgebraic splines that cannot be so represented.

\end{document}